\documentclass[12pt]{amsart}
\usepackage{amssymb,verbatim}

\title[p-adic local dynamics] {Local analytic conjugacy of resonant analytic mappings in two variables, in the non-archimedean setting}

\subjclass[2000]{Primary 32P05}
\keywords{conjugacy, normal form, non-archimedean, p-adic, formal, analytic, holomorphic}

\author{ Adrian Jenkins \and Steven Spallone }
\address{Department of Mathematics, Kansas State University, Manhattan,
KS, 66506} \email{majenkin@math.purdue.edu}
\address{Department of Mathematics, University of Oklahoma, Norman,
OK, 73019} \email{sspallone@math.ou.edu}

\newtheorem{thm}{Theorem}[section]
\newtheorem{theorem}[thm]{Theorem}
\newtheorem{cor}[thm]{Corollary}
\newtheorem{prop}[thm]{Proposition}
\newtheorem{define}[thm]{Definition}

\newtheorem{lemma}[thm]{Lemma}

\newcommand{\nc}{\newcommand}
\nc{\W}{\ensuremath{\underline{\N}}}
\nc{\dmo}{\DeclareMathOperator}
\nc{\ul}{\underline}
\nc{\N}{\mathbb{N}}
\nc{\Q}{\mathbb{Q}}
\nc{\Z}{\mathbf{Z}}
\nc{\F}{\mathcal{F}}
\nc{\R}{\mathbf{R}}
\nc{\C}{\mathbf C}
\nc{\OK}{\mathcal{O}_0^K}
\nc{\eps}{\varepsilon}
\nc{\ra}{\rightarrow}
\nc{\gm}{\gamma}
\nc{\ol}{\overline}
\nc{\av}{\vec{a}}
\nc{\bv}{\vec{b}}
\nc{\cv}{\vec{c}}
\nc{\vv}{\vec{v}}
\nc{\gammav}{\vec{\gamma}}
\nc{\GG}{\mathfrak{G}}
\nc{\res}{res}
\nc{\Mod}{\text{ mod }}
\dmo{\Id}{Id}
\dmo{\SEMI}{semi}
\dmo{\REPEL}{rep}
\dmo{\ord}{ord}
\dmo{\diag}{diag}

\begin{document}

\begin{abstract}
In this note, we consider locally invertible analytic mappings in two dimensions, with coefficients in a non-archimedean field.  Suppose such a map has a Jacobian with eigenvalues $\lambda_1$ and $\lambda_2$ so that $|\lambda_1|>1$ and $\lambda_2$ is a positive power of $\lambda_1$, or that $\lambda_1=1$ and $|\lambda_2|\neq 1$.  We prove that two formal maps with eigenvalues satisfying either of these conditions are analytically equivalent if and only if they are formally equivalent.
\end{abstract}

\maketitle

\section{Introduction}\setcounter{equation}{0}\label{Introduction}

The basic questions of discrete dynamical systems surround the study of iterates of mappings defined in some set. For example, one can consider a global situation, where a set $S$ is fixed, and the goal is to understand the behavior of any map $f:S\rightarrow S$ on all points of $S$. In this situation, the particular set $S$ on which these maps are defined plays a crucial role in understanding their dynamics.

On the other hand, one can also study the {\it local} dynamics of maps. Rather than fixing a particular set $S$ and studying the maps which send $S$ to $S$, one can consider the dynamical properties of some class of maps in a sufficiently small neighborhood of an interesting dynamical point, e.g., a fixed point.  Let us say that $S$ is a subset of some vector space.  Since the theory is purely local, one often assumes that the fixed point is $0$, for convenience. In this setting, the actual neighborhood of $0$ on which such a map is defined is mostly irrelevant (although geometric considerations may still play a large role in the theory). For this reason, one often considers the study of the dynamics of ``germs" of mappings. 

As an example, consider the set of germs of analytic functions $f$ fixing the origin within ${\mathbf{C}}$. Such germs may be written in the form 
\begin{equation*}
f(z)=\sum_{n=m \geq 1}^{\infty }a_{n}z^{n}.
\end{equation*}
 
A natural goal would be the local reduction of such a mapping to a more ``suitable'' form $f_{0}$, which is easier to study, yet retains all of the dynamical properties of the original function $f$. This can be accomplished via a local change of variable, i.e. a map $h$ fixing $0$ which conjugates $f$ to $f_{0}$ within some suitably small neighborhood $U$ of $0$: $h\circ f\circ h^{-1}=f_{0}$. Note that if such an analytic map exists, then obviously we have $h\circ f^{\circ n}\circ h^{-1}=f_{0}^{\circ n}$, and so all dynamical properties of $f$ are preserved by $f_{0}$. Moreover, depending on the regularity of $f$, more subtle data can be gained (for example, if $h$ is analytic, then it will preserve invariant analytic curves, etc.). We do not give a full accounting of this theory (which is vast); the interested reader may, e.g., consider the survey article of Abate \cite{Abate}. Our interest here is to pursue the same questions in higher dimensions, and on different fields. In this note, we will consider the question of analytic equivalence of germs of analytic maps in two variables, within the non-archimedean context.

Before tackling analytic equivalence of germs of analytic mappings, it is often desirable to first determine the ``formal'' equivalence of such germs. Note that germs of analytic maps at $0$ can be expressed via power series with no constant term. If we restrict ourselves to the set $G$ of (locally) invertible analytic germs, then $G$ forms a group. In fact, it is a subgroup of the group of invertible {\it formal} power series. We can thus consider a weaker relation: two formal germs are called formally equivalent if they are conjugate in this larger group. The advantages of considering formal equivalence are numerous: while formal equivalence obviously does not guarantee analytic equivalence, it is easier to study, typically requiring only arithmetic operations. Because of this, formal theory within, e.g., ${\mathbf C^{n}}$ can be carried over to any field of characteristic $0$. Moreover, a robust theory exists, with many ``formal normal forms" (see Section \ref{Formal Theory} for some examples of these forms). The natural question thus arises: what can be said about two formally-equivalent germs $F$ and $G$?

The authors have turned to applying this formal theory in the nonarchimedean case.  The study of non-archimedean dynamics is an active area of research, encompassing both global and local results; see the works of Benedetto (e.g. \cite{Benedetto}) and Rivera-Letelier (e.g. \cite{Rivera-Letelier}). Rather than working in $\C$ or $\R$, one considers the arithmetically-defined field $\Q_p$ or an extension $K$.  These fields play a prominent role in number theory.  However, our interest here is not arithmetic, but rather the gentler analytic theory that such fields provide.  The norms on these fields satisfy the so-called ultrametric inequality 
\begin{equation*}
|x+y| \leq \max (|x|,|y|),
\end{equation*}
and as a consequence, a series $\sum a_n$ converges if and only if the terms $a_n$ approach $0$.  This makes convergence of power series particularly straightforward. 
 
Our project is meant as a continuation of the fundamental work of Hermann and Yoccoz \cite{Herman-Yoccoz}, which treats the case for which the formal normal forms are linear maps.  
We consider the other cases, which occur when the eigenvalues $\lambda_1, \ldots, \lambda_r$ of the Jacobian $DF_0$ of $F$ have special relations, referred to as resonances.  More precisely, a resonance is any relation of the form
\begin{equation}\label{resonance relation}
\lambda _{j}-\lambda_{1}^{i_{1}}\lambda_{2}^{i_{2}}\ldots \lambda _{r}^{i_{r}}=0
\end{equation}
where for $1 \leq k \leq r$,  the $i_{k} \geq 0$ are natural numbers and $\sum_{k}i_{k}\geq 2$.  These differences arise as denominators in attempting to construct the formal maps which conjugate $F$ to its linear part.  Herman and Yoccoz show that if the differences on the left-hand side of (\ref{resonance relation}) are sufficiently far from $0$, then the maps are analytically linearizable.  There are similar theorems in the complex case due to Siegel \cite{Siegel}, Bryuno \cite{Bryuno} and Yoccoz \cite{Yoccoz2}.  On the other hand one finds in \cite{Herman-Yoccoz} a wealth of examples of formally linearizable two-dimension maps which are not analytically linearizable. This is in stark contrast to the one-dimensional case, where formal equivalence is precisely the same as analytic equivalence (\cite{Herman-Yoccoz}, \cite{Jenkins-Spallone}, among others).  

In this paper we prove that for many cases of resonant two-dimensional maps, formal equivalence still implies analytic equivalence, and moreover, one can determine analyticity by measuring the decay of ``denominators'' in an elementary fashion. We focus on the case for which the Jacobian of $F$ has distinct eigenvalues.  By a linear change of coordinates, we may express a typical function (in $r$ variables) as
\begin{equation}\label{typical}
F(x,y)=(\lambda _{1}x+O(2),  \lambda _{2}y+O(2)),
\end{equation}
given by two power series in the variables $x,y$.  Since we consider the case of locally invertible maps, the eigenvalues $\lambda_i$ are nonzero.  

We have two main results, which we express now as one theorem:
\begin{theorem} \label{formal->analytic} Suppose that $F$ and $G$ are mappings of the form (\ref{typical}) with the same eigenvalues $\lambda_1$, $\lambda_2$.  Suppose either of the following hold:
\begin{enumerate}
\item (Attracting/Repelling Case) $\lambda_2=\lambda_1^n$ for some $n \geq 2$, and $|\lambda_i| \neq 1$
\item (Semihyperbolic Case) $\lambda_1=1$ and $|\lambda_2|\neq 1$.
\end{enumerate}
Then $F$ and $G$ are formally equivalent if and only if they are analytically equivalent.
\end{theorem}

We accomplish this by first proving that each such map $F$ is analytically equivalent to its Poincar\'{e}-Dulac (or ``PD") form, which is one of the normal forms we consider. The conjugating maps are inductively constructed by composing homogeneous polynomial maps $P_{n}$ of degree $n$.  Each $P_{n}$ eliminates the corresponding degree-$n$ terms of $F$. In the absence of resonance, $F$ may be formally linearized with this method.

In this paper we analyze the mechanics of the Poincar\'{e}-Dulac algorithm, and measure the growth of the coefficients as the maps $P_{n}$ compose.  The precise rate of growth of these coefficients is subtle, but if captured can prove convergence.   We find estimates on the various $P_{n}$ which hold up under composition.  We refer to such estimates as ``dynamic functions" (the name is chosen because these functions behave well with respect to conjugation and iteration).  In fact, one of the interesting consequences of our construction of dynamic functions is that they determine nontrivial groups of analytic germs. In particular, we show that the various $P_{n}$ lie in a dynamic group, so that their composition $\Phi$ does as well.  This method of incorporating dynamic estimates into the PD-algorithm had its genesis in our one-variable paper \cite{Jenkins-Spallone}.

In the attracting/repelling resonant case, our method proves that $F$ is analytically equivalent to its PD-form
\begin{equation}
F_1(x,y)=(\lambda x, \lambda^n y+Cx^{n}).
\end{equation}
If $C \neq 0$, then this may be further reduced by a linear map to
\begin{equation}
F_0(x,y)=(\lambda x, \lambda^n y+x^{n}).
\end{equation}
  
The two forms $(\lambda x, \lambda^n y+x^{n})$ and $(\lambda x, \lambda^n y)$ are formally inequivalent, and it follows that formal equivalence implies analytic equivalence in this case.
  
In the semihyperbolic case, our method proves that $F$ is analytically equivalent to its PD-form
  
\begin{equation}  \label{semi PD}
F_{0}(x,y)=\left( x+\sum_{i=2}^{\infty }a_{i}x^{i}, \lambda y\left( 1+\sum _{j=1}^{\infty }b_{j}x^{j}\right) \right).
\end{equation}

While this result is of independent interest (in particular, it guarantees that the PD-form is in fact analytic), it does not yet suffice to give Theorem \ref{formal->analytic}.  In an earlier work, Jenkins \cite{Jenkins} showed that (\ref{semi PD}) may be further formally reduced to a certain polynomial form, which we refer to as the PDJ-form.   More precisely,

\begin{theorem} \label{Jenkins Form}
Fix $0\neq \lambda \in K$, and let
\begin{equation*}
F_{0}(x,y)=(f(x),\lambda y(1+g(x))),
\end{equation*}
where $f$ and $g$ are locally $K$-analytic at $0$, with $f(x)=x+\rho x^{m}+O(x^{m+1})$, with $\rho \neq 0$, and $g(0)=0$. Then, there is a polynomial $r(x)$ with $r(0)=0, \deg r<m$ and an analytic mapping $H(x,y)=(h(x),yk(x))$ with $h$ tangent to the identity and $k(0)=1$, and $\rho, \mu \in K$, so that
\begin{equation*}
H\circ F_{0}\circ H^{-1}(x,y) =(x+\rho x^{m}+\mu x^{2m-1}, \lambda y (1+r(x))).
\end{equation*}
\end{theorem}

The PDJ-form is unique up to the action of an $(m-1)$-root of unity on $r(x)$.  In this paper we review this reduction and prove that is actually an analytic conjugation, using another argument akin to that of \cite{Jenkins-Spallone}.  The semihyperbolic case of Theorem \ref{formal->analytic} now follows easily.  

We now describe the layout of this paper.  In Section \ref{Preliminaries} we give a brief summary of the theory of nonarchimedean fields and observe some algebraic structure of the maps we are considering.
 In Section \ref{Formal Theory} we recall the theory of formal equivalence in the form that we need, and clarify some uniqueness of normal forms.
 In Section \ref{attracting and repelling maps!} we prove the attracting/repelling case of Theorem \ref{formal->analytic}.  In Section \ref{semi-hyperbolic maps!} we prove in the semihyperbolic case that maps are analytically equivalent to their PD-form.  In Section \ref{Jenkins!} we prove Theorem \ref{Jenkins Form}, which gives Theorem \ref{formal->analytic} for the semihyperbolic case.
Section \ref{Dynamic!} introduces the notion of a dynamic function and works out the properties that we use in this paper. As dynamic functions not only provide an elementary means of estimating power series (in both one and several variables), but also determine nontrivial subgroups of analytic germs, we believe that this theory will have application beyond this paper.  Some concluding remarks may be found in Section \ref{Remarks}.

\section{Preliminaries and Notation}\setcounter{equation}{0}\label{Preliminaries} 
 
\subsection{Non-archimedean Fields}
In this paper $K$ denotes a non-archimedean complete (nontrivial) valued field $K$ of characteristic $0$. We give only the briefest survey of the pertinent definitions for our purposes.  For a more fleshed-out approach, see \cite{Schikhof}.

\begin{define} Let $K$ be a field.  A non-archimedean norm on $K$ is a map $| \cdot |: K \rightarrow \R$ satisfying the following rules, for all $x,y \in K$:
\begin{enumerate}
\item $|x| \geq 0, |x|=0$ if and only if $x=0$.
\item $|x+y| \leq \max (|x|,|y| )$.
\item $|xy|=|x||y|$.
\end{enumerate}
A non-archimedean field is a pair $(K, | \cdot |)$.

\end{define}

We will simply write $K$ when the norm is implicit.

Perhaps the most important examples of non-archimedean fields are the following: consider $K=\Q$ and choose a prime $p \in {\mathbb Z}$.  Consider the function
\begin{equation*}
\left| \frac{m}{n}\right| _{p}=\left( \frac{1}{p}\right) ^{\ord_p(m)-\ord_p(n)},
\end{equation*}
where $\ord_p(n)$ is the exponent of $p$ in the prime factorization of $n$.
Then $|\cdot |_{p}$ is a non-archimedean norm on $\Q$, and the field of $p$-adic numbers $\Q_p$ is defined to be the topological completion of this normed field. 

Balls in normed fields are defined in the usual way.  For $x_{0}\in K$ and $\eps >0$, we define the ``open'' and ``closed'' balls centered at $x_{0}$ of radius $\eps$ as 
\begin{equation*}
B(x_{0},\eps)=\{ x\in K: |x-x_{0}|<\eps \} ,
\end{equation*}
\begin{equation*}
B^{+}(x_{0},\eps)=\{ x\in K: |x-x_{0}|\leq \eps \} .
\end{equation*}

We recall one estimate which will be critical to the proof of Theorem \ref{Jenkins Form}. A proof may be found in \cite{Schikhof}.

\begin{prop}\label{n!}
Given a field $K$ of characterstic $0$ with non-archimedean norm $|\cdot |$, there is an $\alpha \in {\mathbf R}$ with $\alpha >0$ so that for all natural numbers $n$, we have $|n!|\geq \alpha ^{n}$.
\end{prop}

For example, if ${\mathbb Q}_{p}\subseteq K$, then we may take $\alpha =1/p$.
 
It is worth noting that for any non-archimedean field $K$, the ball $\Delta =B^{+}(0,1)$ is a ring, and for $0<\eps<1$, the ball $B^{+}(0,\eps)$ is an ideal of this ring. We will refer to $\Delta $ as the ring of integers of $K$.   In the case $K=\Q_p$, the ring $\Delta=\mathbb{Z}_p$ is called the ring of $p$-adic integers.

\subsection{Near-rings and groups of  mappings}
Write $\W$ for the set of nonnegative integers.

\begin{define} Given a vector $\vec{a}=(a^{(1)}, \ldots, a^{(r)}) \in \W^r$, let
\begin{equation*}
|\vec{a}|=\sum_{k=1}^r a^{(k)}.
\end{equation*}
\end{define}

\begin{define} Given a formal power series $ f \in K[|x_1,\ldots,x_r |]$, and a vector $\vec{a}=(a^{(1)}, \ldots, a^{(r)}) \in \W^r$, write $[f]_{\av}$ for the coefficient of $x_1^{a^{(1)}} \cdots x_r^{a^{(r)}}$ in $f$.  
\end{define}
Thus,
\begin{equation*}
f(x_1, \ldots, x_r)= \sum_{\av} [f]_{\av} x_1^{a^{(1)}} \ldots x_r^{a^{(r)}}.
\end{equation*}

\begin{define} Let $(x_1, \ldots, x_r) \in K^r$ and let $\eps_1, \ldots, \eps_r$ be positive real numbers.  Then the product of balls $B(x_1,\eps_1) \times \cdots \times B(x_r,\eps_r)$ in $K^r$ is called a polydisc.  A power series $f \in K[|x_1,\ldots,x_r|]$ is analytic (at $0$) if it converges in some polydisc in $K^r$ containing $0$.
\end{define}

As usual, geometric growth of coefficients implies analyticity. 
\begin{lemma} Let $f \in K[|x_1,\ldots,x_r|]$.  Suppose that there is a number $R>0$ so that for all $\vec{a}$ with $|\vec{a}|$ sufficiently large, we have
\begin{equation*}
|[f]_{\vec{a}}| \leq R^{|\vec{a}|}.
\end{equation*}
Then $f$ is analytic at $0$.
\end{lemma}
\begin{proof}
In fact it converges on any polydisc $B(0,\epsilon)^r$, with $0 < \eps<\frac{1}{R}$.

\end{proof}

\begin{define} Fix a number $t \in \W$, and put $A_t=\{ \vec{a} \mid |\vec{a}| \geq t \}$. \end{define}
 
\begin{define} Let $K[|x_1,\ldots,x_r |]_0$ denote the set of formal power series $f$ in $r$ variables with zero constant term, i.e., with $f(0)=0$.  Let $n \geq 1$.  Put
\begin{equation*}
I[n]=\{ f \in K[|x_1,\ldots,x_r |]_0 \mid [f]_{\av}=0 \text{ unless } |\av| \geq n  \}.
\end{equation*}
\end{define}

\begin{define} Let $\F=\F_r$ be the set of formal maps $F: K^r \to K^r$ with $F(0)=0$.  These are given in the usual way by $r$-tuples of power series in $K[|x_1,\ldots,x_r |]_0$.
Let
\begin{equation*}
\F[n]= \{ F \in \F \mid \pi_kF \in I[n] \text{   } \forall k \}.
\end{equation*}
\end{define}

Here $\pi_kF$ denotes the $k$th coordinate of $F$.  We also write  $[F]^{k}_{\vec{a}}$ for the coefficient of  $x_1^{a^{(1)}} \cdots x_r^{a^{(r)}}$ in $\pi_{k}F$.
Of course, $F$ is analytic (at $0$) iff each $\pi_kF$ is analytic.

By the following lemma, whose proof we omit, we will often assume that $\tilde F$ has integral coefficients.  Write $L_q$ for scalar multiplication by an element $q \in K$.
\begin{lemma} \label{q-change} Let $F \in \F$ be analytic.  Then there is a $q \in K^{\times}$ so that
\begin{equation*}
\left[ L_q^{-1} \circ F \circ L_q \right]_{\vec{a}}^k \in \Delta,
\end{equation*}
for all $k$ and for all $\vec{a} \in A_2$.
\end{lemma}

$\F$ is not quite a ring under the operations of addition and composition.  The  left distribution law 
\begin{equation*}
F \circ (G+H)=F \circ G +F \circ H
\end{equation*} 
generally does not hold.  Since all the other axioms of a ring hold, including right distribution, $\F$ is what is called a ``near-ring".  (For a general introduction to near-rings, see \cite{Pilz}.)

Moreover $\F[n]$ is a two-sided ideal of $\F$.  Write $F \equiv G \Mod \F[n]$ if $F-G \in \F[n]$.  It is easy to see that 
$F \circ (G + H) \equiv F \circ G \Mod \F[n]$ if $H \in \F[n]$, and so the set of left additive cosets $\F/\F[n]$ inherits the structure of a near-ring.  Write 
\begin{equation*}
p[n]: \F \to \F/\F[n]
\end{equation*}
for the quotient map, a homomorphism of near-rings.

Next, we introduce some groups associated to these near-rings.

\begin{define}  Write $\diag(\Delta)$ for the group of diagonal linear maps in $\F$ whose eigenvalues are units.
For $F \in \F$, write $DF_0$ for its linear part (the Jacobian at $0$), and let $\tilde{F}=F-DF_0$.
Put 
\begin{equation*}
\GG=\GG_r=\{ F \in \F \mid DF_0 \in \diag(\Delta) \},
\end{equation*}
and 
\begin{equation*}
\GG[n]=\{ F \in \GG_r \mid \tilde{F} \in  \F[n] \}.
\end{equation*}
\end{define}

It is easy to see that $\GG_r$ is a group under composition.  Note that $\GG_1$ is the group of formal power series with unit multiplier.
Since $\GG[n]= p[n]^{-1}(\diag(\Delta)+\F[n])$, it is a subgroup of $\GG_r$.
 
Finally, we define a subset $\GG(\tau)$ of formal maps in $\GG$ whose coefficients are governed by a function $\tau$.

\begin{define} Fix a number $t$, and a function $\tau: A_t \to K^r$ given by $\tau=(\tau_1, \ldots, \tau_r)$ and let
\begin{equation*}
I(\tau_k)=\{ f \in I[t] \mid \tau_k(\av) [f]_{\av} \in \Delta \text{  } \forall \av \in A_t \},
\end{equation*}
\begin{equation*}
\F(\tau)= \{ F \in \F[t] \mid \pi_k F \in I(\tau_k) \text{  } \forall k \}.
\end{equation*}
 
Finally, put
\begin{equation*}
\GG(\tau)= \{ F \in \GG[t] \mid \tilde{F} \in \F(\tau) \}.
\end{equation*}
\end{define}
 
We now give our first examples of groups associated to particular functions $\tau$. 
 
 \begin{prop}\label{maxes}
Let $r=t=2$.
Pick $\lambda \in K$ with $|\lambda| > 1$.  Put $t(m,n)=\max ( 1,n )$.  For both $k=1,2$, define $\tau_k: A_2 \to K$ via $\tau_k(m,n)=\lambda^{t(m,n) }$.  Let $\tau=(\tau_1,\tau_2)$.  Then $\GG(\tau)$ is a subgroup of $\GG$.
\end{prop}
 
\begin{prop} \label{mixed}
Let $r=t=2$.  Pick $\lambda \in K$ with $|\lambda| > 1$.
Fix a  number $n \in \N$.  Define $\tau_k: A_2 \to K$ as follows.  Put $t_1(i,j)=i+nj$, and $t_2(i,j)=\max(n,i+nj)$. For $k=1,2$, let $\tau_k(i,j)=\lambda^{t_k(i,j)}$.  Let $\tau=(\tau_1,\tau_2)$.  Then $\GG(\tau)$ is a subgroup of $\GG$.
\end{prop}

The proofs of these propositions are deferred to Section \ref{Dynamic!}.  These two groups are instrumental to the analysis which follows, but they are by no means the only groups which we can construct. In fact, in Section \ref{Dynamic!}, we identify a class of functions $\tau$, called ``dynamic'', so that if $\tau$ is dynamic, then $\GG(\tau)$ forms a subgroup of $\GG$.
 
The following lemma will be useful later.  Its proof is immediate.

\begin{lemma} \label{useful} Let $r=2$ and $m \geq 2$.  Suppose
\begin{equation*}
H_m(x,y)= \left(x+ \sum_{i+j=m} c_{ij}x^iy^j, y+ \sum_{i+j=m} d_{ij}x^iy^j \right).
\end{equation*}
Let $G \in \F_2$ with $DG_0(x,y)=(\lambda_1 x,\lambda_2 y)$ for some $\lambda_1, \lambda_2 \in K$.  Then for $i+j=m$ we have
\begin{equation*}
\left[ H_m \circ G \right]^1_{(i,j)}=\left[ G \right]^1_{(i,j)}+c_{ij} \lambda_1^i \lambda_2^j,
\end{equation*}
and
\begin{equation*}
\left[ H_m \circ G \right]^2_{(i,j)}=\left[ G \right]^2_{(i,j)}+ d_{ij} \lambda_1^i \lambda_2^j.
\end{equation*}
For $n \geq 2$, $i+j=m$, and $k=1,2$ we have
\begin{equation*}
[ (\pi_k(H_m \circ G))^n]_{(i,j)}= [ (\pi_k G)^n ]_{(i,j)}.
\end{equation*}

\end{lemma}

\section{Formal Equivalence} \label{Formal Theory}

\subsection{One Variable}

We recall here some of the formal theory in one variable, which holds for any field $K$ of characteristic zero.   In the complex setting the formal theory has been known for some time, and is impossible to ascribe to a single source.  The theory is entirely algebraic and needs only minor modifications when $K$ is not algebraically closed.

\begin{prop}\label{rational formal conjugacy}
Suppose that $f(x)=x+\rho x^{m}+\mu x^{2m-1}+ O(x^{2m})$, with $\rho, \mu \in K$ and $\rho \neq 0$.
\begin{enumerate}
\item $f$ is formally equivalent to $f_0(x)=x+ \rho x^m+ \mu x^{2m-1}$.
\item  Suppose that $f$ is formally equivalent to another power series of the form $g(x)=x+\rho'x^{n}+\mu'x^{2n-1}+ O(x^{2n})$.  Then $m=n$, and there exists $c \in K^\times$ so that $c^{m-1}\rho'=\rho$ and $c^{2m-2}\mu'=\mu$. 
\end{enumerate}
\end{prop}
 
Fix sets $R_j$ of coset representatives for $K^{\times}/ (K^{\times})^j$, so that all $\rho \in R_j$ have $|\rho| \leq 1$.  For example, if either $K$ is algebraically closed or if $j=1$, we may pick $R_j=\{ 1 \}$.  If $K=\Q_p$ with $p$ odd then we may pick $R_2=\{1,\epsilon, p, \epsilon p \}$, where $\epsilon \in \mathbb{Z}_p^{\times}$ is not a square.

\begin{define} We say that a map $f$ with multiplier one is in (rational) formal normal form if $f(x)=f_{m,\rho,\mu}(x)=x+\rho x^m+ \mu x^{2m-1}$ with $\rho \in R_{m-1}$.  
\end{define}

From Proposition \ref{rational formal conjugacy}, any map $f$ tangent to the identity is formally conjugate to a unique formal normal form. This choice is convenient for the proof of Theorem \ref{Jenkins Form}. 

\begin{prop} \label{us, bitches!} If $K$ is nonarchimedean, then two one-dimensional maps are analytically equivalent if and only if they are formally equivalent.  In particular each is analytically equivalent to its formal normal form. 
\end{prop}

We refer the reader to  \cite{REU}, \cite{Herman-Yoccoz}, \cite{Jenkins-Spallone} for proofs.

\subsection{Poincar\'{e}-Dulac Theory}

This section gives a quick look at some of the formal normalizations of mappings $F \in \F$. Just as important as the formal results which we will present are the series which are constructed through the theory. In fact, our own results rest entirely on the ability to estimate these coefficients which arise in the construction.

We explain here how to eliminate terms in a function of the form (\ref{typical}). Before we begin, we have the following definition. Let $\gamma _{i}\in K$ for $i=1,\ldots ,n$. We say that the numbers $\{ \gamma _{i}\} $ possess resonance if there is a relation of the form
\begin{equation}\label{resonance}
\gamma _{j}=\prod_{i=1}^{n} \gamma _{i}^{m_{i}},
\end{equation}
where $m_{i}$ is a nonnegative integer for each $i$ and $\sum m_{i}\geq 2$. For example, if $1\geq |\gamma _{1}|\geq \cdots \geq |\gamma _{n}|>0$, then only finitely many resonances can exist. On the other hand, note that any set of the form $\{ 1, \lambda _{1},\ldots ,\lambda _{n}\} $ will possess infinitely many resonances, regardless of the presence of resonances among the set $\lambda _{1},\cdots ,\lambda _{n}$. In particular, we have that
\begin{enumerate}
\item
$1=1^{n}$ for all $n\geq 2$,
\item
$\lambda _{i}=1^{n}\lambda _{i}$ for all $n\geq 1$.
\end{enumerate}
Let $F,G \in \F$  of the form (\ref{typical}).  Our conjugating maps will take the form $H_{n}=Id+P_{n}({\mathbf x})$, where $P_{n}$ is a homogeneous polynomial mapping of degree $n$, $n\geq 2$. Write $\Phi_{n}=H_{n}\circ \cdots\circ H_{2}$, and write $F_{n}=\Phi _{n}\circ F\circ \Phi_{n}^{-1}$, for $n\geq 2$, and $F_{1}=F$. Each $H_{n}$ will be used to eliminate all terms of degree $n$ in $F$, except those which correspond to resonances in the eigenvalues. In order to accomplish this, we determine $H_m$ by the formulae (here, $m=\alpha _{1}+\cdots +\alpha _{n}$):
\begin{equation}\label{Poincare-Dulac formula}
[H_{m}]^{k}_{(\alpha _{1},\ldots ,\alpha _{n})}=\frac{[F_{m-1}]^{k}_{(\alpha _{1},\ldots ,\alpha_{n})}-[G]^{k}_{(\alpha _{1},\ldots ,\alpha _{n})}}{\lambda _{k}-\lambda _{1}^{\alpha _{1}}\cdots \lambda _{n}^{\alpha _{n}}}.
\end{equation}

Of course, this algorithm breaks down if the denominator is zero, i.e., whenever a resonance relation of the form ({\ref{resonance}) exists. Note that if no resonance relations exist, then the map may be formally linearized. If resonances exist, then this procedure will reduce the mapping $F$ to a form $F_{0}$ consisting of a sum of only resonant monomials, i.e. those monomials with multidegrees corresponding to some resonance condition. We will refer to this map $F_{0}$ as the Poincar\'{e}-Dulac (or PD) form of $F$. Note that, {\it a priori}, it is not clear that $F_{0}$ is actually an analytic mapping if infinitely many resonant conditions exist (e.g. if $\lambda _{1}=1$).

We remark that any coefficient of $H_{n}$ which is not used in the simplification (i.e. any coefficient of a term whose multidegree corresponds to a resonance) can be considered a ``free term'', and so if a resonance exists, then the conjugating map taking $F$ to its PD-form is not unique. We will take these free terms to be zero in what follows.\\

\noindent {\bf Example}. Let us suppose that $K$ is any normed field of characteristic $0$, and that a two-dimensional formal map $F$ of the form (\ref{typical}) has eigenvalues $\lambda_1=2$ and $\lambda_2=16$. Since only one resonance relation exists ($\lambda_1^{4}=\lambda_2$), the above algorithm allows us to choose our map $G$ to be of the form $G(x,y)=(2x, 16y+Cx^{4})$. Since the map $G$ is polynomial, it is automatically analytic, and so the actual value of $C$ is mostly irrelevant in that context (and can be altered by a linear change of variable - see Proposition \ref{PD Uniqueness}). From the above algorithm, we obtain a family of formal maps $H$, dependent upon a single parameter (given by a choice of the value of $[H_{4}]^{2}_{(4,0)}$), all of which conjugate $F$ to the normal form $G$.\\ 

\noindent {\bf Example}. Let us suppose again that $n=2$, and that $K$ is any field of characteristic $0$. This time, let $F$ be a formal map of the form (\ref{typical}) with eigenvalues $\lambda_1=1$ and $\lambda_2=\lambda$ which is not a root of unity. As mentioned at the beginning of this subsection, infinitely many resonances are present in this case. Thus, our map $G$ will take the form of an infinite power series with nonzero coefficients $[G]^{1}_{(j,0)}$ and $[G]^{2}_{(k,1)}$. In this case, the Poincar\'{e}-Dulac algorithm yields a family of formal conjugating maps which is dependent on {\it infinitely many} parameters, each corresponding to a choice of resonant monomial. Note again that, {\it a priori}, no analyticity can be assumed for the formal normal form $G$, even when the original map $F$ is analytic.\\

\subsection{Further Normalizations of Poincar\'{e}-Dulac Forms}

For semi-non-resonant maps, the Poincar\'{e}-Dulac algorithm leaves infinitely many terms in each component. However, such maps may be further normalized formally to a polynomial form, as was done by the first author \cite{Jenkins}.
We state here these results in the two-variable case.  Under the appropriate hypothesis, it easily generalizes to $n+1$ variables (namely, if $DF_0=\diag(1,\lambda _{1},\ldots ,\lambda _{n})$, then we assume that no resonances exist between the eigenvalues $\lambda _{1},\ldots ,\lambda _{n}$ - such a set of eigenvalues is called ``essentially non-resonant''). 

We assume that our map $F$ takes the form
\begin{equation}\label{Poincare-Dulac form}
F(x,y)=(f(x),\lambda y(1+g(x))).
\end{equation}
Here $\lambda \neq 0$,  $f$ is tangent to the identity, and $g$ is a formal power series with $g(0)=0$.

Thus, when reducing $F$, we begin by reducing $\pi _{1}F=f$ to its (rational) formal normal form.
 From Proposition \ref{rational formal conjugacy}, choose $\rho \in R_j$, $\mu \in K$, and a power series $h(x)=x + \cdots$, so that if
\begin{equation*}
H_{0}(x,y)=(h(x),y),
\end{equation*}
then 
\begin{equation*}
H_0^{-1}\circ F\circ H_0=F_1,
\end{equation*}
 where
\begin{equation}\label{new Poincare-Dulac form}
F_1(x,y)=(f_{m,\rho,\mu}(x),\lambda y(1+g_1(x))).
\end{equation}
Here  $g_1 =g\circ h\in K[|x|]_0$.  Thus, we may assume that $\pi_1F=f_{m,\rho,\mu}$.

Next is the significant step of reducing the second component to $\lambda y (1+r(x))$, where $r$ is a {\it polynomial} whose degree is less than $m$.  One does this with maps of the form $K(x,y)=(x,yk(x))$, where $k(0)=1$.  This algorithm will be explained in the proof of Theorem \ref{Jenkins Form}.  

\subsection{Uniqueness of Normal Forms}
In this section we treat the uniqueness of the formal normal forms of this paper.

\begin{prop} \label{PD Uniqueness} Let $\lambda \in K$ and $n \geq 2$.   Suppose that $\lambda^n \neq \lambda$. Put $F_0(x,y)=(\lambda x, \lambda^n y+ x^n)$.
\begin{enumerate}
\item Let $F(x,y)=(\lambda x, \lambda^n y+ Cx^n)$ with $C \neq 0$.  Then $F$ is analytically equivalent to $F_0$.
\item $F_0$ is not formally linearizable.
\end{enumerate}
\end{prop}

\begin{proof}
The first statement can be demonstrated with a linear diagonal map, which we leave to the reader to discover.  

We prove the second statement.  Write $L$ for the linear map $L(x,y)=(\lambda x, \lambda^n y)$.
Suppose that $\Phi \in \F_2$ is a  two-dimensional map with $\Phi \circ F_0=L \circ \Phi$.  Taking derivatives shows that $D \Phi_0$ commutes with $L$, and is thus diagonal.

Therefore we may write
\begin{equation*}
\Phi(x,y)=\left(ax+ \sum_{ij} a_{ij} x^i y^j, by+ \sum_{ij} b_{ij} x^i y^j \right),
\end{equation*}
with $a,a_{ij},b,b_{ij} \in K$.
  One finds that
\begin{equation*}
[ \Phi \circ F_0 ]^2_{(n,0)}=b + b_{n0} \lambda^n
\end{equation*}
and 
\begin{equation*}
[L \circ \Phi]^2_{(n,0)}=  b_{n0} \lambda^n.
\end{equation*}
It follows that $b=0$, so in particular $\Phi$ is not invertible.
\end{proof}

\begin{lemma} \label{Amanda} Let $f(x)=x+A x^m+ O(x^{2m-1})$, with $A \neq 0$, and suppose that $h \in K[|x|]_0$ is an invertible series which centralizes $f$.  Then $h(x) \equiv \zeta x \Mod x^m$, where $\zeta$ is an $(m-1)$-root of unity.
\end{lemma}

\begin{proof}
 Let $h(x)=\sum_{n\geq 1}a_{n}x^{n}$, and with $h\circ f=f\circ h$.  Equating the $m$th coefficients gives
 \begin{equation}
 a_1A+a_m=Aa_1^m+a_m,
 \end{equation}
 and so $a_1=\zeta$ is an $(m-1)$-root of unity.  
 
 Next, we prove by induction on $2 \leq n < m$ that $a_n=0$.  Thus, writing $h(x)=\zeta x+ a_n x^n + \ldots$, one computes that
 \begin{equation}
 [h \circ f]_{m+ (n-1)}= a_{m+ (n-1)}+ n a_n A
 \end{equation}
 and
 \begin{equation}
 [f \circ h]_{m+ (n-1)}= a_{m+ (n-1)}+ Am a_n.
 \end{equation}
 This implies that $a_n=0$, as desired.
 
 \end{proof}
 
Note that if $\zeta$ is an $(m-1)$-root of unity, then the linear map $L_{\zeta}(x)=\zeta x$ does indeed centralize $f$.

\begin{define} A map 
\begin{equation*}
F(x,y)=(f(x), \lambda y(1+r(x)))
\end{equation*}
with $\lambda \neq 1$ is in PDJ-normal form if 
\begin{enumerate}
\item $f=f_{m,\rho,\mu}$ is in (rational) formal normal form
\item $\deg r < m$
\item $r(0)=0$.
\end{enumerate}
Write $F=F_{\lambda,f,r}$ for this map.

\end{define}

A PDJ-normal form is almost a formal invariant.

\begin{prop} (Uniqueness of PDJ-normal Form) Let $F$ and $G$ be maps in PDJ-normal form.  
Write $F=F_{\lambda,f,r}$, and $G=F_{\lambda', g,r'}$.  Suppose that $F$ and $G$ are formally equivalent.  Then $\lambda'=\lambda$, $f=g$, and there is an $(m-1)$-root of unity $\zeta$ so that $r(x)=r'(\zeta x)$.
\end{prop}
\begin{proof}
Since $F$ and $G$ are formally equivalent, $\lambda=\lambda'$.
 
Suppose that $\Phi$ is an invertible map with
\begin{equation} \label{amanda}
\Phi \circ F= G \circ \Phi.
\end{equation}
By Lemma 2.1 of \cite{Jenkins}, we may write $\Phi(x,y)=(h(x), y k(x))$ for an invertible power series $h \in K[|x|]_0$ and $k \in K[|x|]$, with $k(0) \neq 0$.  The first component of (\ref{amanda}) gives
$h \circ f= g \circ h$.  Since $f$ and $g$ are in normal form, we have $f=g$ and $h$ centralizes $f$.  Say $f(x)=x+\rho x^m+ \mu x^{2m-1}$, with $\rho \in R_{m-1}$.
The second component of (\ref{amanda}) gives
\begin{equation*}
(1+r(x))(k \circ f)(x)= k(x)(1+(r' \circ h)(x)).
\end{equation*}
Reading this equation modulo $x^m$ gives
\begin{equation*}
(1+r(x))k(x) \equiv k(x) (1+ r'(\zeta x)),
\end{equation*}
where $\zeta$ is an $(m-1)$-root of unity.  We have used the Lemma \ref{Amanda} above.
Since $k(0) \neq 0$, we may multiplicatively invert $k(x)$ modulo $x^m$.  It follows that $r(x) \equiv r'(\zeta x)$ modulo $x^m$.  Since $\deg(r(x)), \deg(r'(\zeta x)) < m$, it follows that they are equal.
\end{proof}
 
\begin{cor} \label{PDJ corollary} Let $F$ and $G$ be maps in PDJ-normal form.  If $F$ and $G$ are formally equivalent, then they are analytically equivalent.
\end{cor}
\begin{proof}  Suppose that $p(x)=q(\zeta x)$ for some $(m-1)$-root of unity $\zeta$. Let $H(x,y)=(\zeta x,y)$.  It is easy to see that
\begin{equation*}
H \circ F_{\lambda, f, r}= F_{\lambda,f,r'} \circ H.
\end{equation*}
\end{proof}

\section{Repelling Fixed Points with Resonance}\setcounter{equation}{0}\label{attracting and repelling maps!}

In this section we prove the Repelling Case of Theorem \ref{formal->analytic}.  We do this by demonstrating that the formal conjugating maps constructed in Section \ref{Preliminaries} are analytic in this case.  The technique here is as follows:   We pick a dynamic function $\tau$ to bound the coefficients of the conjugating maps.  This function will govern the coefficients of all $H_m$; since $\tau$ is dynamic it will automatically govern the coefficients of the composition $\Phi_m$.  The real work is to check that the functional equation at each stage yields the $\tau$-estimate on $H_{m+1}$ from the $\tau$-estimate on $\Phi_m$.  This implies the $\tau$-estimate on the limit $\Phi$, and therefore this limit is analytic.
 
We assume that the map $F=F(x,y)$ is analytic in a neighborhood of $0$ with eigenvalues $\lambda $ and $\lambda ^{n}$ at $0$, where $|\lambda |> 1$ and $n\geq 2$. After a linear change of variable, we can write
\begin{equation*}
F(x,y)=\left(\lambda x+\sum_{i+j\geq 2}a_{ij}x^{i}y^{j},\lambda ^{n}y+\sum_{i+j\geq 2}b_{ij}x^{i}y^{j} \right).
\end{equation*}
In this case, the Poincar\'{e}-Dulac form is polynomial, and takes the form
\begin{equation}\label{attracting Poincare-Dulac form}
F_{0}(x,y)=(\lambda x,\lambda ^{n}y+bx^{n}).
\end{equation}
By the formal theory, through a polynomial change of variable, we may assume that
\begin{equation}\label{attracting prenormal form}
F(x,y)=\left (\lambda x+\sum_{i+j\geq n+1}a_{ij}x^{i}y^{j},\lambda ^{n}y+bx^{n}+\sum_{i+j\geq n+1}b_{ij}x^{i}y^{j} \right).
\end{equation}
Since this series is convergent, we may make a final linear change of variable via Lemma \ref{q-change}, and assume that $a_{ij},b$ and $b_{ij}$ are integers for all $(i,j)$ with $i+j\geq n+1$. (Recall that integers are defined as the elements in $K$ of norm bounded by $1$.)  This assumption will be in place throughout this paragraph. 
We construct  polynomials $H_{m}$  for $m \geq n+1$ through the formula (\ref{Poincare-Dulac formula}).

\begin{define} Let $\GG_{\REPEL}=\GG_T(\tau)$, as defined as in Proposition \ref{mixed}, using this choice of $\lambda$ and $n$.
\end{define}
Recall that for $G \in \GG_{\REPEL}$, we have
\begin{enumerate}
\item $\lambda^{i+nj} [G]^1_{(i,j)} \in \Delta$,
\item $\lambda^{\max(n,i+nj)}[G]^2_{(i,j)} \in \Delta$.
\end{enumerate}

By Proposition \ref{group under composition}, $\GG_{\REPEL}$ is a group.  We will inductively show that each $H_m \in \GG_{\REPEL}$.

\begin{prop}\label{attracting c_{ij} estimate}
Let $F$ be an analytic germ of  the form (\ref{attracting prenormal form}) with  $a_{ij},b$ and $b_{ij}$ integers. Let $H_{m}$  be defined as before for all $m\geq n+1$. Then $H_m \in \GG_{\REPEL}$.
\end{prop}

\begin{proof}
We apply induction on $m$, where $m=i+j$ is the total degree of any term of the homogeneous polynomial map $H_{m}-Id$.  Note that $H_{n+1} \in \GG_{\REPEL}$ (using the Poincar\'{e}-Dulac formula (\ref{Poincare-Dulac formula})).   So assume $H_i \in \GG_{\REPEL}$ for  $n+1 \leq i \leq m$; we argue now that $H_{m+1} \in \GG_{\REPEL}$.
Since $\GG_{\REPEL}$ is closed under composition, we have $\Phi_m \in \GG_{\REPEL}$.
 
We write $\Phi _{m+1}=H_{m+1}\circ \Phi _{m}$, where
\begin{equation*}
H_{m+1}(x,y)=\left(x+\sum_{i+j=m+1}c_{ij}x^{i}y^{j}, y+\sum_{i+j=m+1}d_{ij}x^{i}y^{j} \right),
\end{equation*}
and $\Phi _{m}(x,y)=H_{m}\circ H_{m-1}\circ \cdots \circ H_{n+1}$ is written
\begin{equation*}
\Phi_{m}(x,y)=\left(x+\sum_{i+j\geq n+1} A_{ij}x^{i}y^{j}, y+\sum_{i+j\geq n+1}B_{ij}x^{i}y^{j} \right).
\end{equation*}
Since $\Phi_m \in \GG_{\REPEL}$, we have
\begin{equation} \label{I am the night}
\lambda ^{i+nj}A_{ij}\in \Delta ,\ \ \ \ \ \ \ \ \lambda ^{\max ( n, i+nj) }B_{ij}\in \Delta .
\end{equation}

The coefficients $c_{ij}$ (resp. $d_{ij}$), where $i+j=m+1$, are determined by the equation
\begin{equation*}
[F_0 \circ \Phi_{m+1}]^k_{(i,j)}=[\Phi_{m+1} \circ F]^k_{(i,j)},
\end{equation*}
for $k=1$ (resp. $k=2$).

Let $k=1$.  By Lemma \ref{useful}, we obtain
\begin{equation*}
\lambda[ \Phi_{m}]^1_{(i,j)}+\lambda c_{ij}=[\Phi_{m} \circ F]^1_{(i,j)}+ c_{ij}\lambda^i\lambda^{nj}.
\end{equation*}
Thus,
\begin{equation}\label{repelling c_{ij} equation}
c_{ij}(\lambda -\lambda ^{i+nj})=[\Phi _{m}\circ F-\lambda \Phi _{m}]^{1}_{(i,j)}.
\end{equation}
We want to show that the right-hand side of (\ref{repelling c_{ij} equation}) is always an integer.  Expanding this out gives
\begin{equation*}
\left[(\pi_1F-\lambda x) +\sum A_{\alpha \beta}(\pi_{1}F)^{\alpha }(\pi_{2}F)^{\beta }-\lambda \sum A_{\alpha \beta }x^{\alpha }y^{\beta } \right]_{(i,j)}.
\end{equation*}
 We may write $\pi_1F= \lambda f_1$, and $\pi_2F= \lambda^n f_2$, where $f_1$ and $f_2$ have integer coefficients.  Then for each $\alpha$, $\beta$, we have
 \begin{equation*}
 A_{\alpha \beta}(\pi_{1}F)^{\alpha }(\pi_{2}F)^{\beta }=\lambda^{\alpha+n \beta}A_{\alpha \beta} f_1 f_2,
 \end{equation*}
 which is integral by (\ref{I am the night}).  Since $\alpha+n\beta \geq 1$, we also have
 $\lambda A_{\alpha \beta} \in \Delta$.  Thus this expression is integral and $\lambda^{i+nj} c_{ij} \in \Delta$, as desired.

Let $k=2$.  By Lemma \ref{useful} again, we obtain
\begin{equation*}
\lambda^n [\Phi_{m}]^2_{(i,j)}+ \lambda^n d_{ij}+b [ (\pi_1 \Phi_m)^n]_{(i,j)}= [\Phi_m \circ F]^2_{(i,j)}+ d_{ij} \lambda^i \lambda^{nj}.
\end{equation*}
Thus,
\begin{equation*}
d_{ij}(\lambda ^{n}-\lambda ^{i+nj})=[\pi_{2}(\Phi_{m}\circ F)-\lambda ^{n}\pi_{2}\Phi_{m}-b(\pi_{1}\Phi_{m})^{n}]_{(i,j)}.
\end{equation*}
We can rewrite the right-hand side of the above equation as
\begin{equation*}
[(\pi _{2}F-\lambda ^{n}y)+\sum B_{\alpha \beta }(\pi _{1}F)^{\alpha }(\pi_{2}F)^{\beta }-\lambda ^{n}\sum B_{\alpha \beta }x^{\alpha }y^{\beta }-b(\pi_{1}\Phi_{m})^{n}]_{(i,j)}.
\end{equation*}
The terms are handled as in the previous case, thus $\lambda^{\max(i+jn,n)}d_{ij} \in \Delta$.
\end{proof}

\begin{prop} If $F \in \F_2$ has eigenvalues $\lambda$ and $\lambda^n$ with $n \geq 2$ and $|\lambda| > 1$,  then $F$ is analytically equivalent to its PD-form.
\end{prop}
\begin{proof}
By the previous proposition, we have $H_m$ and therefore $\Phi_m \in \GG_{\REPEL}$ for all $m$.  It follows that $\Phi$ and $\Phi^{-1}$ lie in $\GG_{\REPEL}$.  In particular, the coefficients of these functions are integers.  Thus $F$ is analytically conjugate to its normal form. 
\end{proof}

{\it Proof of Repelling Case of Theorem \ref{formal->analytic}}.
Suppose that $F$ and $G$ are formally equivalent mappings of the form (\ref{typical}) with the same eigenvalues $\lambda$,  and $\lambda^n$, with $n \geq 2$ and $|\lambda| > 1$.    By the previous proposition, we may assume $F(x,y)=(\lambda x, \lambda^n y+ C_1x^n)$ and $G(x,y)=(\lambda x, \lambda^n y+ C_2x^n)$ for some $C_1,C_2 \in K$.  By Proposition \ref{PD Uniqueness}, two such maps are formally equivalent if and only if they are analytically equivalent.  \qed

\section{Semi-hyperbolic mappings}\setcounter{equation}{0}\label{semi-hyperbolic maps!}
We now consider semihyperbolic mappings. This section and the next are devoted to a proof that any map of the form
\begin{equation}\label{diagonal semi-hyperbolic}
F(x,y)=(x+O(2),\lambda y+O(2))
\end{equation}
for which $|\lambda |\neq 1$ is analytically equivalent to its polynomial normal form
\begin{equation}
F_r(x,y)=(f_{m,\rho, \mu }(x),\lambda y(1+r(x)))
\end{equation}
as described in Section \ref{Preliminaries}.

In this section, we will prove that the general semi-hyperbolic mapping (\ref{diagonal semi-hyperbolic}) is analytically equivalent to its PD-form (\ref{Poincare-Dulac form}). The method is similar to that of the previous section.  Later, we show that any mapping in its PD-form is analytically equivalent to its PDJ-form.

We assume that $F$ is an analytic mapping of two variables fixing the origin whose eigenvalues at $0$ are $1$ and $\lambda$ satisfying $1<|\lambda |$.  (For the case $0 < |\lambda| <1$, one may consider the inverse $F^{-1}$.)  After a linear change of variable, we will assume that $DF_0$ is diagonal, and that all higher-degree terms have integer coefficients. That is, we have
\begin{equation}\label{general semi-hyperbolic}
F(x,y)=\left( x+\sum_{j+k=2}^{\infty }a_{jk}x^{j}y^{k}, \lambda y+\sum_{j+k=2}^{\infty}b_{jk}x^{j}y^{k}\right) ,
\end{equation}
 where $a_{ij},b_{ij}\in \Delta $ for all $(i,j)$ with $i+j\geq 2$. 

 Our goal is to show that the formal techniques of the Poincar\'{e}-Dulac theory yield both analytic normal forms as well as analytic intertwining maps.  The PD-form of $F$ has the form $F_{0}(x,y)=(f(x),\lambda y(1+g(x)))$; let us write this as
 \begin{equation}
 F_0(x,y)=\left( x+\sum_{j=2}^{\infty}a_j^0x^j,\lambda y \left(1+ \sum_{k=1}^{\infty} b_k^0 x^k \right)\right).
 \end{equation}
 
 The formal conjugating map $H$ is uniquely determined if we assume that it is tangent to the identity, and that $\pi_{1}H$ has no term of the form $c_{i0}x^{i}$, and $\pi_{2}H$ has no term of the form $d_{j1}x^{j}y$.

\begin{define} Let $\GG_{\SEMI}=\GG_T(\tau)$, as defined as in Proposition \ref{maxes}, using this choice of $\lambda$ and $n$.
\end{define}

Recall that for $G \in \GG_{\SEMI}$, and $k=1,2$, we have
\begin{equation*}
 \lambda^{\max(1,n)}[G]^k_{(i,j)} \in \Delta.
\end{equation*}

By Proposition \ref{group under composition}, $\GG_{\SEMI}$ is a group.

The PD-form $F_0$ adds a complication to the matter;  its construction is interlaced with the construction of $\Phi$.  We will inductively see that its coefficients are integral by the next lemma.  Recall that we construct $H_{m}$ and $\Phi _{m}$ for $m\geq 2$ so that $\Phi_{m}\circ F=F_{0}\circ \Phi_{m}$ modulo terms of total degree no less than $m+1$.  

\begin{lemma}\label{F_{0} coefficients}
Suppose that $F$ is written in the form (\ref{general semi-hyperbolic}) so that $\tilde F$ has integer coefficients. Suppose also that $\Phi_{m} \in \GG_{\SEMI}$. Then $\tilde F_{0}$ has integer coefficients modulo $\F[m+1]$.
\end{lemma}
\begin{proof}

Since $\Phi^{-1}_m \in \GG_{\SEMI}$, we know in particular that $\lambda ( \tilde{\Phi }_{m}^{-1})$ has integer coefficients.  It follows that $(F\circ \Phi _{m}^{-1})(x,y)=(x,\lambda y)+G(x,y)$, where $G$ has integer coefficients.   Since $\Phi_m \in \GG_{\SEMI}$, we have $\Phi_m(x,\lambda y)-(x,\lambda y)$ is integral.  Thus, the function $(\Phi_{m}\circ F\circ \Phi_{m}^{-1})(x,y)-(x,\lambda y)$ has integer coefficients. Since this function agrees with $F_{0}$ modulo $(m+1)$-degree terms, the result is proved.
\end{proof}

\begin{prop} 
Let $F$ be an analytic germ of  the form (\ref{general semi-hyperbolic}) with  $a_{ij}$ and $b_{ij}$ integers. Let $H_{m}$ and $F_0$ be defined as before for all $m\geq n+1$. Then $H_m \in \GG_{\SEMI}$, and $\tilde F_0$ has integer coefficients.
\end{prop}
 
\begin{proof}
We induct on $m$.
Let us first write 
\begin{equation*}
H_{m+1}(x,y)=\left(x+\sum_{i+j=m+1}c_{ij}x^{i}y^{j}, y+\sum_{i+j=m+1}d_{ij}x^{i}y^{j} \right).
\end{equation*}
Recall from the formal algorithm that we choose $c_{m+1,0}=d_{m,1}=0$; this is important for the analysis that follows.
By the formal theory, we have $\Phi _{m+1}\circ F \equiv F_{0}\circ \Phi _{m+1} \mod \F[m+2]$ is satisfied. We will consider the terms of total degree $m+1$ on each side of this equation.

We start by looking at the first components. For any $(i,j)$ with $i+j=m+1$ and $j \neq 0$ (since we
have already defined $c_{m+1,1}=0$), Lemma \ref{useful} yields
\begin{equation}\label{1st LHS}
[\Phi_{m+1}\circ F]^1_{(i,j)}=[\Phi_{m}\circ F]^1_{(i,j)}+c_{ij}\lambda ^{j},
\end{equation}
and
\begin{equation}\label{1st RHS}
\begin{split}
[F_{0}\circ \Phi_{m+1}]^1_{(i,j)} =[\Phi_{m}]^1_{(i,j)}+\sum_{k=2}^{m} a_k^0 \left[(\pi_{1}\Phi_{m})^{k}\right]_{(i,j)}+c_{ij}.
\end{split}
\end{equation}
Equating the two sides gives
\begin{equation}\label{1st}
c_{ij}(1-\lambda ^{j})=[\Phi _{m}\circ F]^1_{(i,j)}-[\Phi_{m}]^1_{(i,j)}-\sum_{k=2}^{m}a_k^0 [(\pi_{1}\Phi_{m})^{k}]_{(i,j)}.
\end{equation}
 
Since $H_i \in \GG_{\SEMI}$ for $i<m+1$ by hypothesis, we have $\Phi_m \in \GG_{\SEMI}$; in particular its coefficients are integral.  Moreover,  $\Phi(x,\lambda y)-(x,\lambda y)$ is integral.  By Lemma \ref{F_{0} coefficients}, we have $a_j^0 \in \Delta$ for $j \leq m$.   Thus the right hand side of (\ref{1st}) is an integer, and it follows that
\begin{equation} \label{combo 1}
\lambda^{\max(1,j)}c_{ij} \in \Delta.
\end{equation}

The proof for the second component is similar. Computing the $(i,j)$-terms for $i+j=m+1$, with $i\neq m$, we have the equations
\begin{equation}\label{2nd LHS}
[\Phi_{m+1}\circ F]^2_{(i,j)}=[\Phi_{m}\circ F]^2_{(i,j)}+d_{ij}\lambda ^{j}.
\end{equation}
and
\begin{equation}\label{2nd RHS}
[F_{0}\circ \Phi_{m+1}]^2_{(i,j)}=\lambda [\Phi_{m}]^2_{(i,j)}+\lambda \sum_{k=1}^m b_k^0 \left[(\pi_{2}\Phi_{m})(\pi_{1}\Phi_{m})^{k}\right]_{(i,j)}+d_{ij}\lambda .
\end{equation}
Setting the two   equal to one another  we obtain
\begin{equation}\label{2nd}
d_{ij}(\lambda -\lambda ^{j})=[\Phi_{m}\circ F]^2_{(i,j)}-\lambda [\Phi_{m}]^2_{(i,j)}- \lambda \sum_{k=1}^m b_k^0 \left[(\pi_{2}\Phi_{m})(\pi_{1}\Phi_{m})^{k}\right]_{(i,j)}
\end{equation}
We argue that the right hand side is again integral.  The first term is integral as in the previous argument.  Let $\phi_k= \pi_k \tilde \Phi_m$.    For $k=1,2$, the polynomials $\lambda \phi_k$ have integral coefficients.  The only possible non-integral term of degree $(i,j)$ in
\begin{equation}
\lambda (y+\phi_2)(x+\phi_1)^k,
\end{equation}
then, is for $k=m$ and thus $(i,j)=(m,0)$.   However, we have already defined $d_{m1}=0$, so the estimate holds trivially in this case.  (Notice the left hand side is $0$.)
It follows that
\begin{equation} \label{combo 2}
\lambda^{\max(1,j)}d_{ij} \in \Delta.
\end{equation}

Combining (\ref{combo 1}) and (\ref{combo 2}) gives the proposition.

\end{proof}

\begin{prop} \label{P-D analytic} If $F \in \F_2$ has eigenvalues $\lambda$ and $1$ with $\lambda$ not a root of unity, then $F$ is analytically equivalent to its PD-normal form.
\end{prop}

\begin{proof}
By the previous proposition, we have $H_m$ and therefore $\Phi_m \in \GG_{\SEMI}$ for all $m$.  It follows that $\Phi$ and $\Phi^{-1}$ lie in $\GG_{\SEMI}$.  In particular, the coefficients of these functions are integers.  Thus $F$ is analytically conjugate to its normal form.  
\end{proof}

\section{Further Normalization}\setcounter{equation}{0}\label{Jenkins!}

The previous section demonstrated that given a map of the form
\begin{equation}\label{semihyperbolic map 2}
F(x,y)=(x+O(2),\lambda y +O(2)),
\end{equation}
where $\lambda \in K$ satisfies $|\lambda |\neq 1$, there is an analytic change of variable $H(x,y)=(x+O(2),y+O(2))$ which conjugates $F$ to the simpler map
\begin{equation}\label{poincare-dulac form 2}
F_{0}(x,y)=(f(x),\lambda y(1+g(x))),
\end{equation}
where $f$ is a one-variable analytic map which is tangent to the identity, and $g$ is an analytic map satisfying $g(0)=0$. We now turn to further normalization, following the theory outlined in Section \ref{Preliminaries}. We will prove a slightly more general result here. Fix $\lambda _{i}\in K$ with $\lambda _{i}\neq 0$ for $i=1,\ldots ,n$; they may or may not possess resonance. Consider an analytic function of the form
\begin{equation}\label{(n+1)-dimensional map}
F(x,y_{1},\ldots ,y_{n})=(f(x),\lambda _{1} y_{1}(1+g_{1}(x)),\ldots ,\lambda _{n}y_{n}(1+g_{n}(x)),
\end{equation}
where $f$ is a one-variable analytic map tangent to the identity and $g _{i}\in K[|x|]_0$.  By Proposition \ref{us, bitches!}, there is a locally-analytic map $h=h(x)$ fixing $0$ so that $h\circ f\circ h^{-1}(x)=f_{m,\rho ,\mu }$, and so by conjugating $F$ in (\ref{(n+1)-dimensional map}) by the change of variable $H$ which is $h$ in the first component and the identity in all other components, we may assume that $f=f_{m,\rho, \mu}$.

Let $L$ be the diagonal map with eigenvalues $(a,1,\ldots ,1)$. An easy check shows that 
\begin{equation*}
\begin{split}
& (L\circ F\circ L^{-1})(x,y_{1},\ldots ,y_{n})\\
& =(a f(x/a ),\lambda _{1}y_{1}(1+g_{1}(x/a )),\ldots ,\lambda _{n}y_{n}(1+g_{n}(x/a ))).\\
\end{split}
\end{equation*}
Thus, by choosing $a$ with sufficiently large norm, we may assume that the coefficients of $f$ and $g_{i}$ are small. In particular, we will assume that all of these coefficients have norm less than or equal to $1$.

Our goal is to conjugate $F$ to the reduced form
\begin{equation}\label{analytic normal form}
F_r(x,y_{1},\ldots ,y_{n})=(f_{m,\rho,\mu}, \lambda y _{1}(1+r_{1}(x)),\ldots ,\lambda _{n}y_{n}(1+r_{n}(x))),
\end{equation}
where $r_{i}(x)$ is the remainder of $g_{i}(x)$ upon division by $x^m$.
 
The conjugation will be by an {\it a priori} formal map 
\begin{equation*}
\Gamma=\lim_{n\to \infty } \Gamma _{n}=\lim_{n\to \infty } (J_{n}\circ J_{n-1}\circ \cdots \circ J_1), 
\end{equation*}
where 
\begin{equation}
J_{i}(x,y_{1},\ldots ,y_{n})=(x,y_{1}(1+c_{i,1}x^{i}),\ldots ,y_{n}(1+c_{i,n}x^{i})).
\end{equation}
The chosen form of the conjugating map is useful. First, note that for any $i=1,2,\ldots $, the inverse $J_{i}^{-1}$ (as well as the inverse $\Gamma _{i}^{-1}$) is easily determined. Moreover, when conjugating any $F$ of the form (\ref{(n+1)-dimensional map}), the coefficients $c_{i,j}$ are used to reduce the analytic function $\pi_{j}F$ for $j=1,\ldots ,n$. The upshot is that we can isolate and eliminate terms of the functions $g_{i}$ individually, and thus, for ease of notation and with no loss of generality, we will assume from now on that $n=1$, i.e. that the map $F=F(x,y)$.

Theorem \ref{Jenkins Form} will be proven once we ascertain that $\Gamma$ is analytic, which is equivalent to proving that the infinite product  $\prod_{i=1}^{\infty}(1+c_ix^i)$ is analytic. The coefficients $c_i$ are determined inductively.  As above, put $\Gamma _{n}=J_{n}\circ \cdots \circ J_{1}$, and let $\Gamma_0$ be the identity.   We choose $c_n$ so that 
\begin{equation}
[ F \circ \Gamma_n -\Gamma_n \circ F_r ]^2_{(m+n-1,1)}=0.
\end{equation}

We will inductively show that $n! \rho ^{n}c_n \in \Delta$. This estimate will also yield a similar estimate for $\Gamma _{n}$, as the following lemma demonstrates:

\begin{lemma}\label{product estimation}
Suppose that $c_i \in K$ satisfy  $i!\rho ^{i}c_i\in \Delta $ for $1 \leq i \leq k$.
Write
\begin{equation}
\prod_{i=1}^k (1+c_ix^i)=1+\sum_{j\geq 1}A_{j}x^{j}.
\end{equation}
Then we have $n!\rho ^{n}A_n \in \Delta $ for all $n\geq 1$.
\end{lemma}
\begin{proof}
 We can write $A_n=\sum_{|\ul{i}|=n} \alpha _{\ul{i}}c_{\ul{i}}$, where $c_{\ul{i}}=c_{i_{1}}\ldots c_{i_{l}}$ and $\alpha _{\ul{i}} \in \Z$. By hypothesis,
\begin{equation}
i_{1}!\ldots i_{l}!\rho ^{|\ul{i}|}c_{\ul{i}}\in \Delta.
\end{equation}
Since the the quotient $\dfrac{n!}{i_{1}!\ldots i_{l}!}$ is an integer,
we have $n!\rho ^{n}c_{\ul{i}}\in \Delta$, and the lemma follows.
\end{proof}
 
{\it Proof of Theorem \ref{Jenkins Form}.}

We assume that $F$ is in PD-form (\ref{poincare-dulac form 2}), and that the coefficients of $f$ and $g$ are each integral. We now reveal how the coefficients $c_n$ are formed and inductively prove that they satisfy the estimate $n! c_n \in \Delta$ for all $n$.  By Lemma \ref{product estimation} and Proposition \ref{n!}, these estimates will ultimately ensure that the function $\Gamma$ is analytic.
 
We consider the congruence
\begin{equation} \label{save princess}
\pi_2(\Gamma_n \circ F) \equiv \pi_2(F_r \circ \Gamma_n)  \Mod x^{m+n},
\end{equation}
for $n \geq 0$.  This condition governs the choice of $c_n$.

Since $g$ and $r$ are congruent $\Mod x^m$, (\ref{save princess}) holds for $n=0$.
 To determine $c_{n+1}$, we look at the second components of
\begin{equation}
J_{n+1} \circ \Gamma_n \circ F {\text{ and }}F_r \circ J_{n+1} \circ \Gamma_n.
\end{equation}
We choose $c_{n+1}$ so that the following equality holds, $\Mod x^{m+n+1}$:
\begin{equation}\label{junk}
(1+g(x))(1+c_{n+1}(f_{m,\rho,\mu}(x))^{n+1})(1+(\alpha \circ f_{m,\rho,\mu})(x))=
(1+r(x))(1+c_{n+1}x^{n+1})(1+\alpha(x)).
\end{equation}
Here $\alpha(x)=\sum_{j\geq 1}A_{j}x^{j}$ is defined by $\Gamma_n(x,y)=(x,y(1+\alpha(x)))$.

We expand both sides of (\ref{junk}). After subtracting the left side from the right, and considering only the $(m+n)$-degree terms, we solve for $c_{n+1}$.   
First, note that
\begin{equation*}
 (f_{m,\rho,\mu}(x))^{n+1} \equiv x^{n+1} + (n+1) \rho x^{m+n} \Mod x^{m+n+1}.
 \end{equation*}
Because of this, we have 
\begin{equation*}
[c_{n+1} (f_{m,\rho,\mu}(x))^{n+1} ]_{m+n}=(n+1)\rho c_{n+1}.
\end{equation*}
In addition,
\begin{equation*}
[g(x)(f_{m,\rho,\mu}(x))^{n+1}-r(x) x^{n+1}]_{m+n}=0,
\end{equation*}
since $g(x) \equiv r(x) \Mod x^m$.
Since $\alpha$ has no constant term, we have
\begin{equation*}
(\alpha \circ f_{m,\rho,\mu})(x) \equiv \alpha(x) \Mod x^m.
\end{equation*}
Thus,
\begin{equation*}
[c_{n+1} (f_{m,\rho,\mu}(x))^{n+1}(\alpha \circ f_{m,\rho,\mu})(x)-c_{n+1}x^{n+1}\alpha(x)]_{m+n}=0,
\end{equation*}
By combining the above ideas, we conclude
\begin{equation*}
[c_{n+1} g(x) f_{m,\rho,\mu}(x)^{n+1}(\alpha \circ f_{m,\rho,\mu})(x)-c_{n+1}r(x) x^{n+1}\alpha(x)]_{m+n}=0.
\end{equation*}
Therefore the $(m+n)$-coefficient obtained from (\ref{junk}) is the sum of $(n+1)\rho c_{n+1}$ and
\begin{equation}\label{fundamental fact}
[r(x)+\alpha (x)+r(x)\alpha (x)-(g(x)+(\alpha \circ f_{m,\rho,\mu})(x)+g(x)(\alpha \circ f_{m,\rho,\mu})(x))]_{m+n}.
\end{equation}
At this point, we have completed the formal theory; the $(m+n)$-coefficient must be $0$, and this determines $c_{n+1}$ uniquely for each $n \geq 0$. Note that $c_{n+1}$ is well-defined for each $n$, as it depends only on a finite number of terms.  

To settle analyticity, we estimate this coefficient. In light of Lemma \ref{product estimation}, we will prove the estimate 
\begin{equation}\label{c_{n+1} estimate}
|(n+1)!c_{n+1}|\leq 1.
\end{equation}
Since (\ref{fundamental fact}) is equal to $-(n+1)\rho c_{n+1}$, we must show that the product of $n!\rho ^{n} $ with (\ref{fundamental fact}) is integral. 

We will prove this inductively. The reasoning is elementary and mostly combinatorial. We will estimate pieces of (\ref{fundamental fact}) individually; the estimate will hold for the sum of these pieces because our norm is non-archimedean. 

For example, note that the difference $(r(x)-g(x))$ certainly satisfies our estimate, since our coefficients are integral, and both $n!$ and $\rho $ are integral. We will break the remainder of the argument into two steps in order to estimate the remaining terms of (\ref{fundamental fact}).\\ \\
\noindent{\it{Step 1}}\\

Let us consider here the difference $\alpha\circ f_{m, \rho, \mu }-\alpha (x)$. We can write
\begin{equation*}
(\alpha\circ f_{m, \rho, \mu })(x) - \alpha(x)=\sum_{j\geq 1}A_{j}((x+\rho x^{m}+\mu x^{2m-1})^{j}-x^{j}).
\end{equation*}
We will expand the powers on the right-hand side as follows: consider each of the terms 
\begin{equation*}
\binom{j}{a,b,c}x^{a}(\rho x^{m})^{b}(\mu x^{2m-1})^{c}A_j.
\end{equation*} 
Since we are only considering the coefficient of the $(m+n)$-degree term, we must have
\begin{equation}\label{(a,b,c) 1}
a+b+c=j, 
\end{equation}
and
\begin{equation}\label{(a,mb,(2m-1)c) 1}
a+mb+(2m-1)c=m+n.
\end{equation} 
Furthermore, we see that $b$ and $c$ are not simultaneously $0$, or else the term is trivial. Lemma \ref{product estimation} yields
\begin{equation}\label{A_{j} estimate}
j!\rho^{j}A_{j}\in \Delta .
\end{equation}

 There are two possibilities to consider: either $j=n+1$ or $j\leq n$. If $j\leq n$, then it is clear that $|\rho ^{n} |\leq |\rho ^{j}|$ and that $|n!|\leq |j!|$. Thus, because $A_{j}$ satisfies (\ref{A_{j} estimate}), it immediately satisfies the stronger estimate
\begin{equation*}
|n!\rho ^{n}A_{j}|\leq 1.
\end{equation*}
On the other hand, if $j=n+1$, then we must have $(a,b,c)=(n,1,0)$.  This triple corresponds to the single term $(n+1)\rho A_{n+1}x^{m+n}$.  From Lemma \ref{product estimation}, we have
\begin{equation*}
(n+1)!\rho^{n+1}A_{n+1} =n!\rho ^{n}((n+1)\rho A_{n+1})\in \Delta.
\end{equation*} 
Thus, our estimate is satisfied.\\ \\

\noindent{\it{Step 2}}\\

We will now consider the difference $r(x)\alpha (x)-g(x)(\alpha\circ f_{m,\rho,\mu})(x)$. By definition of $r$, we have that $g(x)=r(x)+\tilde g(x)$, where $\tilde g(x)=O(x^{m})$. Using this, we may rewrite this difference as 
\begin{equation*}
r(x)\alpha (x) -r(x)(\alpha \circ f_{m,\rho ,\mu })(x)-\tilde g(x)(\alpha \circ f_{m,\rho,\mu })(x).
\end{equation*}

We begin by analyzing the series $\tilde g(x)(\alpha \circ f_{m,\rho,\mu })(x)$. The $(m+n)$-degree coefficient of this series will be a sum of terms of the form
\begin{equation}\label{2nd difference}
[\tilde g]_{d}[\alpha \circ f_{m,\rho,\mu }]_{m+n-d}.
\end{equation}
Since $[g]_{d}=[\tilde g]_{d}$ for all $d\geq m$, each term in (\ref{2nd difference}) is in turn a sum of integer multiples of terms of the form 
\begin{equation*}
[g]_{d} x^{d}A_{j}x^{a }(\rho x^{m})^b(\mu x^{2m-1})^{c}, 
\end{equation*}
where $[g]_{d}$ is integral, and $d\geq m$. We also have
\begin{equation}\label{(a,b,c) 2}
a+b+c=j,
\end{equation}
but now 
\begin{equation}\label{(a,mb,(2m-1)c,d) 1}
a+mb+(2m-1)c+d=m+n.
\end{equation} 

We use the same techniques as in Step 1. Since $[g]_{d}$ is an integer, we need only show
\begin{equation}\label{Step 2 goal}
n! \rho^n (\rho^b A_j) \in \Delta.
\end{equation}
Again, we already have the estimate (\ref{A_{j} estimate}) for $A_{j}$. Using the fact that $d\geq m$, we can conclude that $j\leq n$, and so our estimate follows as in Step 1.

In similar fashion, we consider $r(x)\alpha (x)-r(x)(\alpha \circ f_{m,\rho,\mu} )(x)$. Writing this out, we obtain a sum of terms which are products of the form
\begin{equation*}
\left[ \sum_{i=1}^{m-1}a_{i}x^{i} \right]_{d} \left[ \sum_{n\geq 1}A_{j}((x+\rho x^{m}+\mu x^{2m-1})^{j}-x^{j}) \right]_{m+n-d}.
\end{equation*}
We look at all expansions of the above product, each of which take the form $a_dx^{d}A_{j}x^{a}(\rho x^{m})^{b}(\mu x^{2m-1})^{c}$. Again 
\begin{equation*}
a+b+c=j,
\end{equation*} 
and
\begin{equation*}
a+mb+(2m-1)c+d=m+n.
\end{equation*}
If $d=0$, the expansion is necessarily trivial. On the other hand, if $1\leq d\leq m-1$, then putting together the equalities above, we see that $j\leq n$. From the estimate (\ref{A_{j} estimate}), the result now follows. \qed \\ \\
 
{\it Proof of Semihyperbolic Case of Theorem \ref{formal->analytic}.}
Suppose that $F$ and $G$ are formally equivalent mappings of the form (\ref{typical}) with the same eigenvalues $\lambda$ and $1$, with $|\lambda |\neq 1$.

By the previous proposition, we may assume that $F$ and $G$ are in PDJ-normal form.  By Corollary \ref{PDJ corollary}, two such maps are formally equivalent if and only if they are analytically equivalent.  \qed

 \section{Dynamic Functions}\setcounter{equation}{0}\label{Dynamic!}

 \subsection{Definition and Examples}
We want to set up estimates for the coefficients of power series, so that if two power series satisfy these estimates, their composition does as well.  For simplicity, let us consider the one-dimensional case.  Suppose we have a sequence of elements $\tau(n) \in K$ for $n \geq t$ and two power series $G(x)=x+\sum_{a \geq t} g_ax^a$ and $H(x)=x+\sum_{b \geq t} h_b x^b$ satisfying $\tau(n)g_n \in \Delta$ and $\tau(n)h_n \in \Delta$ for all $n \geq t$.  For which sequences $\tau(n)$ will the composition $G \circ H$ necessarily satisfy these same estimates?

We have
\begin{equation}
(G \circ H)(x)= H(x)+\sum_{a \geq t} g_a \left(x+\sum_{b \geq t} h_b x^b \right)^a.
\end{equation}
Consider one of these terms for a fixed $a$; we expand it using the multinomial theorem.  A typical term of this expansion involves a choice of a partition $a=a_0+ \cdots + a_s$.  Here the term $x$ is raised to the power $a_0$, and $h_{b_i}x^{b_i}$ is raised to the power $a_i$.  The corresponding monomial is
\begin{equation}
\binom{a}{a_0, \ldots, a_s} g_a \prod_{i >0} h_{b_i}^{a_i}x^c,
\end{equation}
where $c= \sum_{i=0}^s a_ib_i$ with $b_0=1$.  This term will satisfy the $\tau$-estimate if
\begin{equation} \label{instructional}
\tau(c) \binom{a}{a_0, \ldots, a_s} g_a \prod_{i>0} h_{b_i}^{a_i} \in \Delta.
\end{equation}
We already know that 
\begin{equation}
\tau(a) g_a \prod_{i>0}\tau(b_i)^{a_i}  h_{b_i}^{a_i} \in \Delta,
\end{equation}
and so (\ref{instructional}) will follow if
\begin{equation}
\left| \tau(c)  \binom{a}{a_0, \ldots, a_s} \right| \leq \left| \tau(a)\prod_{i>0} \tau(b_i)^{a_i} \right|.
\end{equation}
This last equation is the condition for $\tau$ to be (weakly) dynamic; we now generalize this condition to any dimension.

Recall that $A_t$ is the set of vectors in $\W^r$ with norm greater or equal to a given number $t$.  We will use the notation $\hat e_k$ to denote the standard basis vector which is $1$ at the $k$th component and $0$ at the other components.

\begin{define}  \label{belle} We say that a function $\tau: A_t \ra K^r$ given by $\tau=(\tau_1,\ldots,\tau_r)$, is weakly dynamic if it satisfies
the following property:  Let $\av=(a^{(1)},\cdots, a^{(r)}) \in A_t$, and suppose there is a partition $a^{(k)}=\sum_{i=0}^{s_k} a_i^{(k)}$
for each $k$, with each $a_i^{(k)} \in \W$, and $a_i^{(k)} \geq 1$ for $i>0$.  Choose for each $k=1,\ldots,r$ and $i=1, \ldots, s_k$ a vector $\bv^{(k)}_i \in A_t$, and put $\bv^{(k)}_0=\hat{e_k}$.  Let
\[ \cv= \sum_{k=1}^r \sum_{i=0}^{s_k}  a_i^{(k)}   \bv^{(k)}_i. \]
Then
\[     |\tau_n(\cv)| \left| \prod_{k=1}^r \binom{a^{(k)}}{a_0^{(k)},\ldots,a_{s_k}^{(k)}} \right| \leq |\tau_n(\av)|\cdot \prod_{k=1}^r \prod_{i=1}^{s_k}
 |\tau_k(\bv^{(k)}_i)|^{a_i^{(k)}}. \]
\end{define}

Here is the prime example of a weakly dynamic function:

\begin{prop}\label{factorial}
Let $r=t=1$ and $\tau(n)=n!$.  Then $\tau$ is weakly dynamic.
\end{prop}

\begin{proof}

We must show that, given partitions $a=\sum_{i=1}^s a_i$, positive integers $b_1, \ldots, b_s$, and $b_0=1$, that
\[ \left| \binom{a}{a_0,\ldots,a_s} \tau(\sum_{i=0}^s a_i b_i) \right| \leq \left|\tau(\sum_i a_i) \right| \left|\prod_{i=0}^s \tau(b_i)^{a_i} \right|. \]
(Note that $\tau(b_0)=1$.)

In other words we must show that the product
\[ \binom{a}{a_0,\ldots,a_s} \frac{(\sum_i a_i b_i)!}{a! \prod_i (b_i)!^{a_i}} \]
is in $\Delta$.  In fact we will show that it is in $\Z$.  Since $\prod_i (a_ib_i)!$ divides $(\sum_i a_ib_i)!$, it is enough to show that for every $i$, the quotient
\[ \frac{ (a_ib_i)!}{a_i! (b_i)!^{a_i}} \]
is an integer.  This is a well-known combinatorial fact. 
\end{proof}

\begin{define} \label{beller} We say that a function $\tau: A_t \ra K^r$ is dynamic if in the above situation it satisfies the stronger inequality
\begin{equation} \label{Beller}
     |\tau_n(\cv)| \leq |\tau_n(\av)|\cdot \prod_{k=1}^r \prod_{i=1}^{s_k} |\tau_k(\bv^{(k)}_i)|^{a_i^{(k)}},
     \end{equation}
for each $1 \leq n \leq r$.
\end{define} 
The dynamic functions for this paper are of the form   $\tau(\av)=(\lambda^{t_1(\av)}, \ldots,\lambda^{t_r(\av)} )$ with $|\lambda|>1$.   Then (\ref{Beller}) translates to
 \begin{equation} \label{t dynamic}
t_n(\cv) \leq t_n(\av)+ \sum_{k=1}^r \sum_{i=1}^{s_k}  a_i^{(k)} t_k(\bv^{(k)}_i).
\end{equation}

\begin{prop}\label{maxes2}
Let $r=t=2$ so that $A_t= \{(m,n) \in \W^2 \mid m+n \geq 2 \}$.
Pick $\lambda \in K$ with $|\lambda| \geq 1$.  Put $t(m,n)=\max ( 1,n )$.  For both $k=1,2$, define $\tau_k: A_2 \to K$ via $\tau_k(m,n)=\lambda^{t(m,n) }$.  Then $\tau=(\tau_1,\tau_2)$ is dynamic.
\end{prop}

Write $b_i^{(k),2}$ for the second component of $\bv_i^{(k)}$.

\begin{proof}
We must verify (\ref{t dynamic}), which is
\[ \max \left( 1, \sum_{k=1}^2 \sum_{i=0}^{s_k} a_i^{(k)} b_i^{(k),2} \right) \leq \max ( 1, a^{(2)} )+ \sum_{k=1}^2 \sum_{i=1}^{s_k} a_i^{(k)} \max ( 1, b_i^{(k),2} ). \]
This is clear when the left hand side is $1$.  Otherwise it reduces to the inequality $a_0^{(2)} \leq a^{(2)}$, which is true.
\end{proof}

\begin{prop} \label{mixed2}
Let $r=t=2$ so that $A_t= \{(m,n) \in \W^2 \mid m+n \geq 2 \}$.  Pick $\lambda \in K$ with $|\lambda| \geq 1$.
Fix a  number $n \in \N$. Put $t_1(i,j)=i+nj$, and $t_2(i,j)=\max(n,i+nj)$. For $k=1,2$, let $\tau_k(i,j)=\lambda^{t_k(i,j)}$.  Then $\tau=(\tau_1,\tau_2)$ is dynamic.
\end{prop}

\begin{proof}
We verify (\ref{t dynamic}), for $k=1,2$.
Let $k=1$.  Note that $t_1(\vv) \leq t_2(\vv)$ for all $\vv$, and $t_1$ is linear.  Thus (\ref{t dynamic}) reduces to
\begin{equation*}
a_0^{(1)}+na_0^{(2) }\leq t_1(\av)=a^{(1)}+na^{(2)},
\end{equation*}
which is clear.
Let $m=2$.  If $t_2(\cv)=t_1(\cv)$, then this follows as in the case $m=1$.  If $t_2(\cv)=n$, then the result follows since then $ t_2(\cv) \leq t_2(\av)$.

\end{proof}
 
Recall the definitions of $\F(\tau)$ and $\GG(\tau)$ from Section \ref{Preliminaries}.  We will now show that if $\tau$ is weakly dynamic, then $\GG(\tau)$ is a group. 

\begin{lemma} \label{dynamic lemma}
Suppose that $\tau$ is weakly dynamic.  If $F \in \F(\tau)$ and $H \in \GG(\tau)$, then $F \circ H \in \F(\tau)$.
\end{lemma}

\begin{proof}
We must show that for all $n$, the expression
\begin{equation} \label{gray owl}
\pi_n F \circ (DH_0+\tilde{H})
\end{equation}
is in $I(\tau_n)$.

For simplicity, let $f= \pi_nF \in I(\tau_n)$.  We have
\begin{equation*}
f(\vec{x})=\sum_{\av \in T} [f]_{\av} x_1^{a^{(1)}} \cdots x_r^{a^{(r)}}.
\end{equation*}
Here we use the notation $\av=(a^{(1)},\ldots, a^{(r)})$.
Writing $\lambda_1, \ldots, \lambda_r$ for the eigenvalues of $DH_0$, 
expression (\ref{gray owl}) is now
\begin{equation*}
\sum_{\av \in T} [f]_{\av} (\lambda_1x_1+ \pi_1 \tilde{H})^{a^{(1)}} \cdots (\lambda_rx^r+\pi_r \tilde{H})^{a^{(r)}}.
\end{equation*}
It is enough to show that each term appearing in the product of multinomial expansions is in $I(\tau_n)$. 
A typical term is formed by choosing a partition $a^{(k)}=\sum_{i=0}^{s_k} a^{(k)}_i$ of each exponent.   Here $\lambda_kx_k$ is raised to the power $a^{(k)}_0$.  One also chooses $s_k$ monomials 
\begin{equation}
[\pi_k \tilde{H}]_{\bv^{(k)}_i}x^{\bv^{(k)}_i}
\end{equation}
from the terms of $\pi_k \tilde{H}$, which will be raised to the power $a^{(k)}_i$.
This term is then given by the product of
\begin{equation}
 \prod_{k=1}^r \binom{a^{(k)}}{a_0^{(k)},\ldots,a_{s_k}^{(k)}} 
\end{equation}
with
\begin{equation} \label{new mac}
 [f]_{\av}  \prod_{k=1}^r \prod_{i=1}^{s_k}[\pi_k \tilde{H}]_{\bv^{(k)}_i}^{a_i^{(k)}} x^{\cv},
 \end{equation}
 where
\begin{equation}
 \cv= \sum_{k=1}^r \sum_{i=0}^{s_k}  a_i^{(k)}   \bv^{(k)}_i. 
\end{equation}
Since $f \in I(\tau_n)$ and $\pi_k \tilde{H} \in I(\tau_k)$, we know that the product of 
\begin{equation}
 \tau_n(\av)\cdot \prod_{k=1}^r \prod_{i=1}^{s_k} \tau_k(\bv^{(k)}_i)^{a_i^{(k)}}
 \end{equation}
 with (\ref{new mac}) is integral.  To show that this typical term is in $I(\tau_n)$, we must show that the product of
 \begin{equation}
 \tau_n(\cv) \prod_{k=1}^r \binom{a^{(k)}}{a_0^{(k)},\ldots,a_{s_k}^{(k)}} 
 \end{equation}
 with (\ref{new mac}) is integral.  But this exactly follows from the definition that $\tau$ is weakly dynamic.

\end{proof}

\begin{thm} Suppose that $\tau$ is dynamic.  Then $\GG(\tau)$ is a subgroup of $\GG$. \label{group under composition}
\end{thm}

\begin{proof}

Let $G,H \in \GG(\tau)$.  We have
\begin{equation}
\begin{split}
G \circ H &= (DG_0+ \tilde{G}) \circ H \\
& \equiv DG_0 \circ H \Mod \F(\tau) \text{, by Lemma \ref{dynamic lemma}} \\
&= D(G \circ H)_0+ DG_0 \circ \tilde{H}. \\
\end{split}
\end{equation}

It is easy to see that $DG_0 \circ \tilde{H} \in \F(\tau)$, and it follows that $G \circ H \in \GG(\tau)$.  

Next, suppose that $G \in \GG(\tau)$, and write $H$ for its inverse in $\GG$.  

For $n \in \N$, we define $H_n \in \F$ via
\begin{equation}
[H_n]_{\av}= \begin{cases}
                         [H]_{\av}  & \text{ if $|\av| \leq n$} \\
                         0 & \text{ otherwise.}
 \end{cases}
\end{equation}
 In other words, we truncate $H$ by only admitting terms of total degree at most $n$.  We will prove by induction that $H_n \in \GG(\tau)$, it certainly follows that $H \in \GG(\tau)$.
 Clearly $H_1=DH_0 \in \GG(\tau)$, and $H_n \in \GG_S$ for all $n$.  
  
 So assume that $H_n \in \GG(\tau)$ for a given $n$.  Let $K_{n+1}=H_{n+1}-H_n$.  We have
 \begin{equation}
 \begin{split}
 I & \equiv H_{n+1} \circ G \Mod \F[n+2] \\
    & =I+ DH_0 \circ \tilde{G} + \tilde{H}_n \circ G+ K_{n+1} \circ G.\\
   \end{split}
   \end{equation}
   Clearly $DH_0 \circ \tilde{G} \in \F(\tau)$, and by Lemma \ref{dynamic lemma}, $\tilde{H}_n \circ G \in \F(\tau)$.  It follows that 
   \begin{equation}
   K_{n+1} \circ G \in \F[n+2]+ \F(\tau).
   \end{equation}
 Note that since all the terms of $K_{n+1} $ have total degree $n+1$, we have $K_{n+1} \circ G \equiv K_{n+1} \circ DG_0 \Mod \F[n+2]$.  It follows that $K_{n+1} \circ DG_0 \in \F[n+2] + \F(\tau)$, and so $K_{n+1}  \in \F[n+2]+\F(\tau)$ and indeed $K_{n+1} \in \F(\tau)$.  It follows that $H_{n+1} \in \F(\tau)$ and we are done by induction.
 \end{proof}

{\bf Remark:} There are other useful one-dimensional dynamic functions.  We mention one which pertains to the work in \cite{Jenkins-Spallone} and \cite{REU}.  Suppose that $f$ is an analytic function of one variable of the form
\begin{equation}
f(x)=x+x^m+\sum_{n=2m-1}^{\infty }a_{n}x^{n}.
\end{equation}
Since $f$ is analytic, there is a $q\in K$, with $0<|q|\leq 1$ so that for all $n\geq 2m-1$, we have $q^{n}a_{n}\in \Delta $. Put $S=\{ m+1, \ldots \}$, and define
\begin{equation}
\sigma(n)=(n+1)+m\left[ \frac{n-2}{m-1}\right] .
\end{equation}
If we write $\tau(n)=q^{\sigma(n)}$, then $\tau$ is dynamic. This function figures prominently in our earlier estimation of formal conjugating power series in one variable. In \cite{REU} other dynamic functions are discussed, which lead to stronger convergence results.

\section{Concluding Remarks} \label{Remarks}
The theorems here show again that in many cases, the formal theory coincides with the analytic theory in the non-archimedean setting, a radical departure from the theory in ${\bf C}^{n}$. Moreover, even in higher dimensions and in the presence of resonance, the estimation of power series remains a useful tool in determining the analyticity of a formal conjugating map $F$.

There are obviously some problems that remain however. The most interesting to the authors is the following: let $F$ be a saddle-hyperbolic map, defined and locally analytic in the neighborhood of some fixed point. In other words, after transporting the problem to $0$, write $F(x,y)=(\lambda _{1}x+O(2),\lambda _{2}y+O(2))$, where $0<|\lambda _{1}|<1<|\lambda _{2}|$. We are interested in the resonant case here, so suppose that some resonance exists.  For the sake of this discussion, let the field be $\Q_{2}$, and let $\lambda _{1}=2$ and $\lambda _{2}=1/2$. Following the Poincar\'{e}-Dulac algorithm, the formal normal form is easily determined; it takes the form
\begin{equation}
F_{0}(x,y)=\left(2x+x\sum_{m=1}^{\infty }a_{m}(xy)^{m}, \frac{1}{2}y+y\sum_{n=1}^{\infty }b_{n}(xy)^{n} \right).
\end{equation}
The problem is then simply put: is $F$ analytically equivalent to $F_{0}$? While it would seem that the techniques of this paper (and specifically, those of Section \ref{semi-hyperbolic maps!}) could be suitably modified to handle this case, the functions which ostensibly govern the growth of the coefficients in the formal conjugating maps are not dynamical. It may be that such maps are simply not analytically conjugated to their formal normal forms, thus providing other examples of formally conjugate maps which are not analytically equivalent.  

Finally, the theory of Section \ref{Dynamic!} is an elementary way to construct tangible, non-trivial subgroups of the set of analytic mappings whose linear part is the identity. These mappings also possess strong estimates dictated by the dynamic functions used to define these groups. In future studies, we hope to understand more fully the structure and properties of these dynamic groups in both one and several dimensions.

\bibliographystyle{plain}

\end{document}